\documentclass[leqno]{article}
\usepackage[cp1250]{inputenc}
\usepackage{authblk}
\usepackage[english]{babel}

\usepackage{a4wide}
 
\usepackage{amsfonts}
\usepackage{graphicx}
\usepackage{epstopdf}
\usepackage{bm}
\usepackage{subfig}
\usepackage{amsthm} 
\usepackage{amsmath}
\usepackage{adjustbox}

 \newtheorem{theorem}{Theorem}

\newtheorem{corollary}{Corollary}
 
 \DeclareMathOperator{\expo}{e}
\DeclareMathOperator{\Id}{Id}

\begin{document}

\title{\textbf{Can classical Schwarz methods for time-harmonic elastic waves converge?
}} \author[1]{Romain Brunet \thanks{E-mail: romain.brunet@strath.ac.uk}}
\author[1,3]{Victorita Dolean \thanks{E-mail: victorita.dolean@strath.ac.uk}}
\author[2]{Martin J. Gander \thanks{E-mail: martin.gander@unige.ch}}
\affil[1]{\small \textit{Department of Mathematics and Statistics, University of Strathclyde, 26 Richmond Street, G1 1XH Glasgow, United Kingdom}}
\affil[2]{\small \textit{Section des Math\'ematiques, Universit\'e de Gen\`eve, 2-4 rue du Li\`evre, CP 64, CH-1211 Gen\`eve, Switzerland}}
\affil[3]{\small \textit{University C\^ote d'Azur, CNRS, LJAD, 06108 Nice Cedex 02, France}}

\date{\today}

\maketitle

\abstract*{We show that applying a classical Schwarz method to the
  time harmonic Navier equations, which are an important model for
  linear elasticity, leads in general to a divergent method for low to
  intermediate frequencies. This is even worse than for Helmholtz and
  time harmonic Maxwell's equations, where the classical Schwarz
  method is also not convergent, but low frequencies only stagnate,
  they do not diverge. We illustrate the divergent modes by numerical
  examples, and also show that when using the classical Schwarz
  method as a preconditioner for a Krylov method, convergence
  difficulties remain.}

\section{Mathematical model}\label{sec:1}

The propagation of waves in elastic media is a problem of undeniable
practical importance in geophysics. In several important applications
- e.g. seismic exploration or earthquake prediction - one seeks to
infer unknown material properties of the earth's subsurface by sending
seismic waves down and measuring the scattered field which comes back,
implying the solution of inverse problems. In the process of solving
the inverse problem (the so-called "full-waveform inversion") one
needs to iteratively solve the forward scattering problem. In
practice, each step is done by solving the appropriate wave equation
using explicit time stepping. However in many applications the
relevant signals are band-limited and it would be more efficient to
solve in the frequency domain. For this reason we are interested here
in the time-harmonic counterpart of the Navier or Navier-Cauchy
equation (see \cite[Chapter 5.1]{eringen1977elastodynamics} or
\cite[Chapter 9]{lautrup2011physics}\footnote{For the fascinating
  history on how Navier discovered the equation and then rapidly
  turned his attention to fluid dynamics, see
  \cite{darrigol2002between}.}), which is a linear mathematical model
for elastic waves given by
\begin{equation}\label{NavierEq}
  -\left(\Delta^e+\omega^2\rho\right)\vec{u}=\vec{f}\quad \mbox{in
    $\Omega$},\quad
  \Delta^e\vec{u}=\mu\Delta\vec{u}+(\lambda+\mu)\nabla(\nabla\cdot\vec{u}),
\end{equation}
where $\mathbf{u}$ is the displacement field, $\mathbf{f}$ is the
source term, $\rho$ is the density that we assume real,
$\mu,\lambda\in [\mathbb{R}_+^*]^2$ are the Lam\'e coefficients,
and $\omega$ is the time-harmonic frequency for which we are
  interested in the solution. An example of a discretization of this
equation was presented in \cite{huttunen2004ultra}.  In our case, we
assumed small deformations which lead to linear equations and consider
isotropic and homogeneous materials which implies that the physical
coefficients are independent of the position and the
direction. Due to their indefinite nature, the Navier equations in
  the frequency domain \eqref{NavierEq} are notoriously difficult to
  solve by iterative methods, especially if the frequency $\omega$
  becomes large, similar to the Helmholtz equation
  \cite{ernst2012difficult}, and there are further complications as
we will see. We study here if the classical Schwarz method could
  be a candidate for solving the time harmonic Navier equations
  \eqref{NavierEq} iteratively.

\section{Classical Schwarz Algorithm} \label{sec:2}

To understand the convergence of the classical Schwarz algorithm
applied to the time harmonic Navier equations \eqref{NavierEq}, we
study the equations on the domain $\Omega:={\mathbb R}^2$, and
decompose it into two overlapping subdomains
$\Omega_1:=(-\infty,\delta)\times{\mathbb R}$ and
$\Omega_2:=(0,\infty)\times{\mathbb R}$, with overlap parameter
$\delta \ge 0$. The classical parallel Schwarz algorithm computes
for iteration index $n=1,2,\ldots$
\begin{equation}\label{ClassicalSchwarz}
  \begin{array}{rcll}
    -\left(\Delta^e+\omega^2\rho\right)\vec{u}_1^n&=&\vec{f}& \mbox{in $\Omega_1$},\\
    \vec{u}_1^n&=& \vec{u}_2^{n-1}&\mbox{on $x=\delta$},\\
    -\left(\Delta^e+\omega^2\rho\right)\vec{u}_2^n&=&\vec{f}& \mbox{in $\Omega_2$},\\
    \vec{u}_2^n&=&\vec{u}_1^{n-1}&\mbox{on $x=0$}.
  \end{array}
\end{equation}
We now study the convergence of the classical parallel Schwarz
  method \eqref{ClassicalSchwarz} using a Fourier transform in the $y$
  direction. We denote by $k \in \mathbb{R}$ the Fourier variable and
  $\hat{\vec{u}}(x,k)$ the Fourier transformed solution,
\begin{equation}
  \begin{array}{rcccl}
    \hat{\vec{u}}(x,k)&=&\mathfrak{F}(\vec{u})&=&\displaystyle
    \int_{-\infty}^{\infty} \expo^{-\mathrm{i} k y} \vec{u}(x,y) \, \mathrm dy,\\
    \vec{u}(x,y) &=& \mathfrak{F}^{-1}(\hat{\vec{u}}) &=&\displaystyle
  \frac{1}{2\pi} \int_{-\infty}^{\infty} \expo^{\mathrm{i} k y}
  \hat{\vec{u}}(x,k) \, \mathrm dk.
  \end{array}
\end{equation}
\begin{theorem}[Convergence factor of the classical Schwarz algorithm]\label{th:convclas}
For a given initial guess $(\vec{u}_1^0 \in(L^2(\Omega_1)^2)$,
$(\vec{u}_2^0 \in(L^2(\Omega_2)^2)$, each Fourier mode $k$ in the
classical Schwarz algorithm \eqref{ClassicalSchwarz} converges with
the corresponding convergence factor
\begin{equation*}
  \rho_{cla}\left(k,\omega,C_p,C_s,\delta\right) = \max \{|r_+|,|r_-|\},
\end{equation*}
where
\begin{equation}\label{r+r-}
r_\pm = \frac{X^2}{2}+e^{-\delta(\lambda_1+\lambda_2)} \pm\frac{1}{2} \sqrt{X^2\left(X^2+4e^{-\delta(\lambda_1+\lambda_2)}\right)},\,
  X = \frac{k^2+\lambda_1\lambda_2}{k^2-\lambda_1\lambda_2}\left(e^{-\lambda_1 \delta}-e^{-\lambda_2 \delta}\right).
\end{equation}
Here, $\lambda_{1,2} \in \mathbb{C}$ are the roots of the
characteristic equation of the Fourier transformed Navier equations,
\begin{equation}
  \lambda_1 = \sqrt{k^2-\frac{\omega^2}{C_s^2}},\quad 
  \lambda_2 = \sqrt{k^2-\frac{\omega^2}{C_p^2}},\quad
C_p=\sqrt{\frac{\lambda+2\mu}{\rho}}, \quad
C_s=\sqrt{\frac{\mu}{\rho}}.
  \label{lambda12}
\end{equation}
\end{theorem}
\begin{proof}
By linearity it suffices to consider only the case $\vec{f}=0$ and
analyze convergence to the zero solution, see for example
\cite{gander2006optimized}. After a Fourier transform in the $y$
direction, \eqref{NavierEq} becomes
\begin{equation}\label{eq:Fourier}
\left\lbrace
\begin{aligned}
&\left[\left(\lambda+2\mu\right)\partial_x^2 + \left(\rho \omega^2-\mu k^2\right)\right] \hat{u}_x + ik(\mu + \lambda)\partial_x \hat{u}_z = 0,\\ 
	&\left[\mu \partial_x^2 + \left(\rho \omega^2-\left(\lambda + 2\mu\right) k^2\right)\right] \hat{u}_z + ik\left(\mu+\lambda\right) \partial_{x} \hat{u}_x  = 0.
\end{aligned}
\right.
\end{equation}
This is a system of ordinary differential equations, whose solution is
obtained by computing the roots $r$ of its characteristic equation,
\begin{equation}\label{eq:caract}
\left[\begin{array}{cc} 
(\lambda+2\mu)r^2 + \rho\omega^2-\mu k^2 & ik(\mu+\lambda)r \\
ik(\mu+\lambda)r & \mu r^2 + \rho\omega^2-(\lambda+2\mu) k^2 
\end{array}\right] \left[\begin{array}{c} \hat u_x \\ \hat u_z \end{array}\right]=0.
\end{equation}
A direct computation shows that these roots are $\pm\lambda_1$ and
$\pm\lambda_2$ where $\lambda_{1,2}$ are given by \eqref{lambda12}.
Therefore the general form of the solution is
\begin{equation}
\hat {\bf u} (x,k) = \alpha_1 {\bf v}_+ e^{\lambda_1 x} +\beta_1 {\bf v}_- e^{-\lambda_1 x} + \alpha_2 {\bf w}_+ e^{\lambda_2 x} + \beta_2 {\bf w}_- e^{-\lambda_2 x},
\end{equation}
where ${\bf v}_{\pm}$ and ${\bf w}_{\pm}$ are obtained by successively
inserting these roots into \eqref{eq:caract} and computing a
  non-trivial solution,
\begin{equation}
\vec{v}+ = \begin{pmatrix} 1\\ \frac{\mathrm{i} \lambda_1}{k} \end{pmatrix},\quad{\bf v}_- = \begin{pmatrix} 1\\ -\frac{\mathrm{i}\lambda_1}{k} \end{pmatrix} ,\quad{\bf w}_+ = \begin{pmatrix} - \frac{\mathrm{i} \lambda_2}{k}\\1 \end{pmatrix},\quad{\bf w}_- = \begin{pmatrix} \frac{\mathrm{i} \lambda_2}{k}\\1 \end{pmatrix}.
\end{equation}
Because the local solutions must remain bounded and outgoing at
infinity, the subdomain solutions in the Fourier transformed domain
are
\begin{equation}\label{eq:locsol}
\vec{\hat{u}}_1 (x,k) = \alpha_1 {\bf v}_+ e^{\lambda_1 x} + \alpha_2 {\bf w}_+ e^{\lambda_2 x},\qquad
\vec{\hat{u}}_2 (x,k) = \beta_1 {\bf v}_- e^{-\lambda_1 x} + \beta_2 {\bf w}_- e^{-\lambda_2 x}.
\end{equation}
The coefficients $\alpha_{1,2}$ and $\beta_{1,2}$ are then uniquely
determined by the transmission conditions. Before using the iteration
to determine them, we rewrite the local solutions at iteration $n$
in the form
\begin{equation}
\begin{aligned}
\vec{\hat{u}}_1^n &= \alpha_1^n {\bf v}_+ e^{\lambda_1 x} + \alpha_2^n {\bf w}_+ e^{\lambda_2 x}
= \begin{bmatrix}
\expo^{\lambda_1 x} &  -\frac{i\lambda_2}{k}\expo^{\lambda_2 x} \\
\frac{i\lambda_1}{k}\expo^{\lambda_1 x} & \expo^{\lambda_2 x}
\end{bmatrix}
\begin{pmatrix}
\alpha_1^n \\
\alpha_2^n
\end{pmatrix}
=: M_x \boldsymbol{\alpha}^n, \\
\vec{\hat{u}}_2^n &= \beta_1^n {\bf v}_- e^{-\lambda_1 x} + \beta_2^n {\bf w}_- e^{-\lambda_2 x}
= \begin{bmatrix}
\expo^{-\lambda_1 x} & \frac{i\lambda_2}{k} \expo^{-\lambda_2 x} \\
- \frac{i\lambda_1}{k} \expo^{-\lambda_1 x} & \expo^{-\lambda_2 x}
\end{bmatrix}
\begin{pmatrix}
\beta_1^n \\
\beta_2^n
\end{pmatrix}
=: N_x \boldsymbol{\beta}^n.
\end{aligned}
\label{formefinaleu}
\end{equation}
We then insert (\ref{formefinaleu}) into the interface iteration
of the classical Schwarz algorithm (\ref{ClassicalSchwarz}),
\begin{equation}
\begin{aligned}
  &M_{\delta} \boldsymbol{\alpha}^{n} = N_{\delta} \boldsymbol{\beta}^{n-1}
  \quad\Longleftrightarrow\quad \boldsymbol{\alpha}^{n} = M^{-1}_{\delta} N_{\delta} \boldsymbol{\beta}^{n-1},  \\
  &N_0 \boldsymbol{\beta}^n = M_0 \boldsymbol{\alpha}^{n-1}
  \quad\Longleftrightarrow\quad \boldsymbol{\beta}^n = N^{-1}_0 M_0 \boldsymbol{\alpha}^{n-1}.
\end{aligned}
\label{pluginter}
\end{equation}
This leads over a double iteration to
\begin{equation}
\begin{aligned}
& \boldsymbol{\alpha}^{n+1} = (M^{-1}_{\delta}N_{\delta}N^{-1}_0 M_0) \boldsymbol{\alpha}^{n-1} =: R_\delta^1 \boldsymbol{\alpha}^{n-1},\\
&\boldsymbol{\beta}^{n+1} = (N^{-1}_0 M_0 M^{-1}_{\delta}N_{\delta}) \boldsymbol{\beta}^{n-1} =: R_\delta^2 \boldsymbol{\beta}^{n-1},
\end{aligned}
\label{SpectralProof}
\end{equation}
where $R_\delta^{1,2}$ are the iteration matrices which are spectrally
equivalent. The iteration matrix $R_\delta^1$ is given by
\begin{equation}
R_\delta^1 = \begin{bmatrix}
\displaystyle\expo^{- \delta(\lambda_1 +\lambda_2)} X_{2}^2 \frac{\lambda_1}{\lambda_2} + \expo^{-2\lambda_1 \delta} X_{1}^2 & \displaystyle X_{1} X_{2} \left(\expo^{-2 \lambda_1 \delta}-\expo^{-\delta (\lambda_1+\lambda_2)}\right) \\
\displaystyle X_{1} X_{2} \frac{\lambda_1}{\lambda_2} \left(\expo^{-\delta (\lambda_1 + \lambda_2)} - \expo^{-2 \lambda_2 L}\right) & \displaystyle\expo^{-\delta (\lambda_1 + \lambda_2)} X_{2}^2 \frac{\lambda_1}{\lambda_2} + \expo^{-2 \lambda_2 \delta} X_{1}^2
\end{bmatrix},
\label{RL1}
\end{equation}
where
\begin{equation}
X_{1} = \dfrac{k^2+\lambda_1 \lambda_2}{k^2-\lambda_1 \lambda_2},
\quad
X_{2} = - \mathrm{i}\dfrac{2 k \lambda_2}{k^2-\lambda_1 \lambda_2}.
\end{equation}
A direct computation then leads to the eigenvalues $(r_+,r_-)$ of
$R_\delta^1$,
\begin{equation}\label{eq:r12}
r_\pm = \frac{X^2}{2}+\expo^{-\delta(\lambda_1+\lambda_2)} \pm \frac{1}{2}\sqrt{X^2\left(X^2+4e^{-\delta(\lambda_1+\lambda_2)}\right)}, \quad X = \dfrac{k^2+\lambda_1 \lambda_2}{k^2-\lambda_1 \lambda_2} \left(\expo^{-\lambda_1\delta}-\expo^{-\lambda_2\delta}\right).
\end{equation}
The convergence factor is therefore given by the spectral radius of
the matrix $R_\delta^{1,2}$,
\begin{equation}
\rho_{cla}\left(k,\omega,C_p,C_s,\delta\right) = \max \{|r_+|,|r_-|\},
\end{equation}
which concludes the proof.
\end{proof}
\begin{corollary}[Classical Schwarz without Overlap]\label{cor}
In the case without overlap, $\delta=0$, we obtain from
  \eqref{r+r-} that $r_{\pm}=1$, since
$\left(R_\delta^1=\Id\right)$.  Therefore, the classical Schwarz
algorithm is not convergent without overlap, it just stagnates.
\end{corollary}
The result in Corollary \ref{cor} is consistent with the general
  experience that Schwarz methods without overlap do not converge, but
  there are important exceptions, for example for hyperbolic problems
  \cite{Dolean2008Why}, and also optimized Schwarz methods can converge
  without overlap \cite{gander2006optimized}. Unfortunately also with
  overlap, the Schwarz method has difficulties with the time harmonic
  Navier equations \eqref{NavierEq}:
\begin{corollary}[Classical Schwarz with Overlap]
\label{ClassicalSchwarzTheorem}
The convergence factor of the overlapping classical Schwarz method
(\ref{ClassicalSchwarz}) with overlap $\delta$ applied to the Navier equations
(\ref{NavierEq}) verifies for $\delta$ small enough
$$
\rho_{cla}\left(k,\omega,C_p,C_s,\delta\right) 
\left\{\begin{array}{l}
= 1,\, k\in\left[0,\frac{\omega}{C_p}\right]\cup\left\{\frac{\omega}{C_s}\right\}, \\ 
> 1,\, k\in\left(\frac{\omega}{C_p},\frac{\omega}{C_s}\right), \\ 
< 1,\, k\in\left(\frac{\omega}{C_s},\infty\right).
\end{array}\right.
$$
It thus converges only for high frequencies, diverges for
  medium frequencies, and stagnates for low frequencies.
\end{corollary}
\begin{proof}
The proof is too long and technical for this short manuscript
    and will appear in \cite{Brunet:2018:OSM}, see also
    \cite{Brunet:2018}. We illustrate however the result for two 
    examples in Figure \ref{Fig1}, which clearly shows the three
    zones of different convergence behavior.
\begin{figure}[t]
  \centering
  \includegraphics[width=0.49\textwidth,clip]{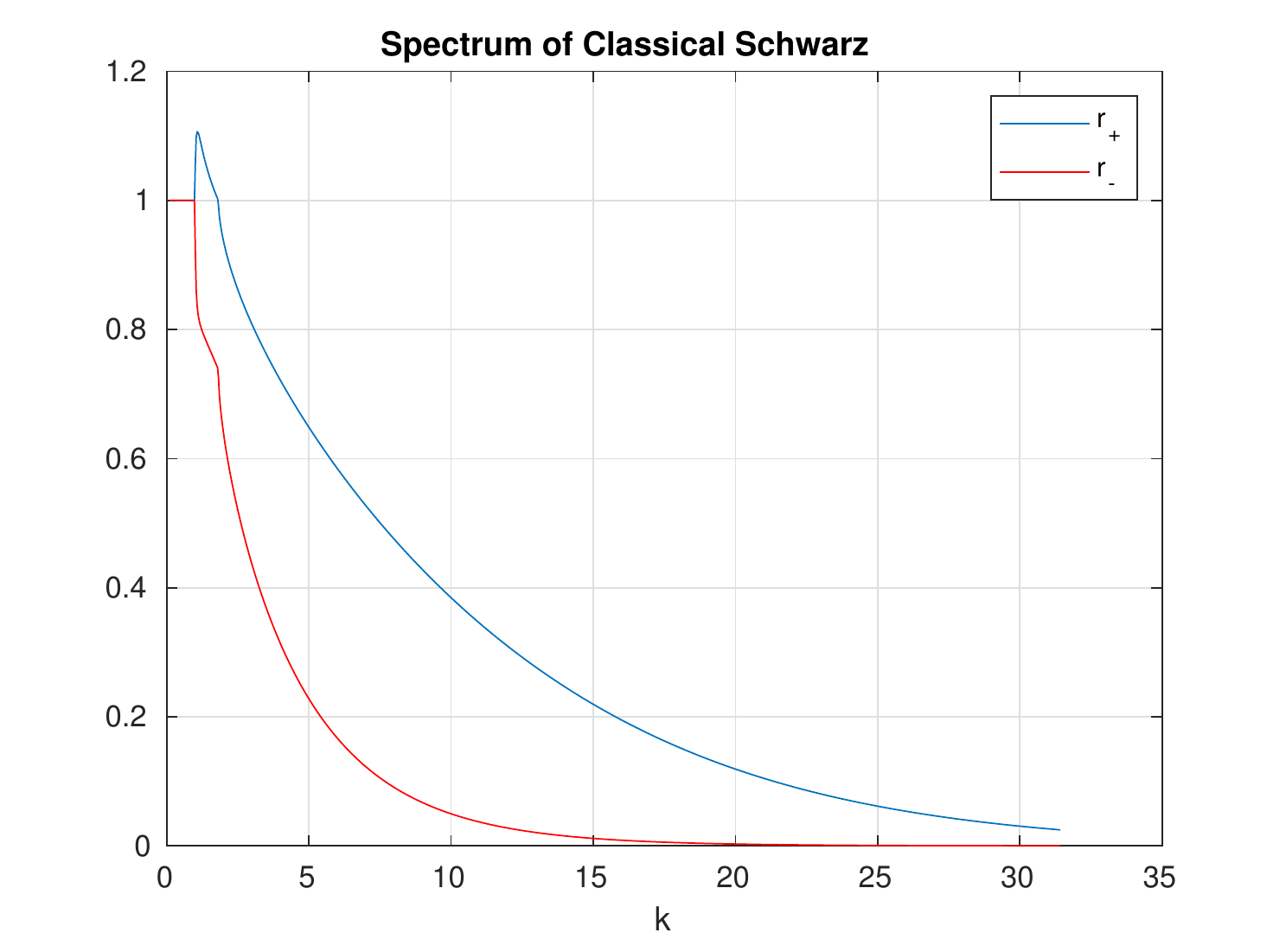}
  \includegraphics[width=0.49\textwidth,clip]{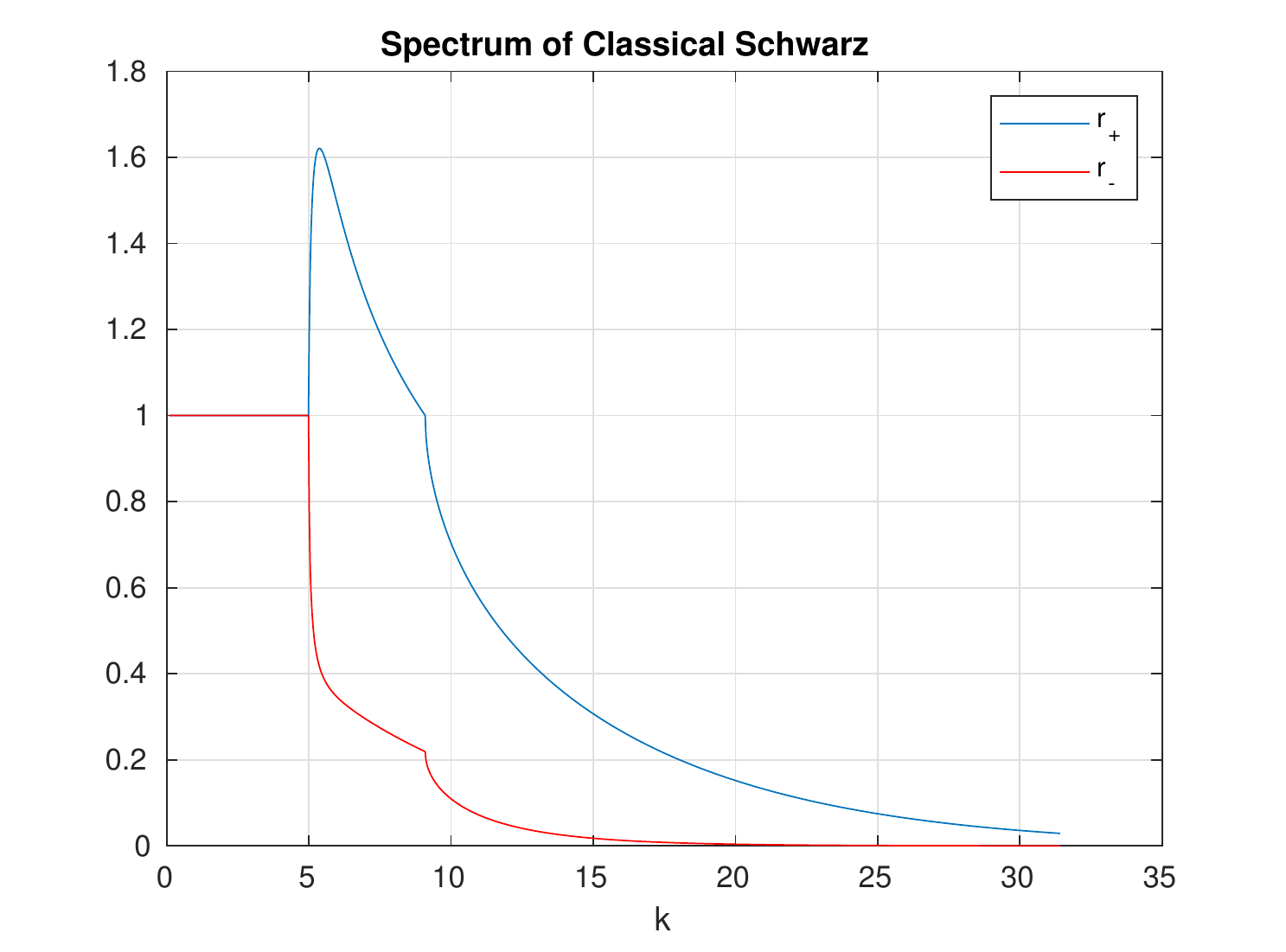}
  \caption{Modulus of the eigenvalues of the iteration matrix as a
      function of Fourier frequency for the classical Schwarz method
    with $C_p=1$, $C_s=0.5$, $\rho=1$, $\delta=0.1$. Left: for
    $\omega=1$. Right: for $\omega=5$.}
  \label{Fig1}
\end{figure}
\end{proof}
We see from Corollary \ref{ClassicalSchwarzTheorem} that the
classical Schwarz method with overlap can not be used as an iterative
solver to solve the time harmonic Navier equations, since it is
in general divergent on the whole interval of intermediate
  frequencies
  $(\frac{\omega}{C_p},\frac{\omega}{Cs})$. This is even
  worse than for the Helmholtz or Maxwell's equations where the
overlapping classical Schwarz algorithm is also convergent
  for high frequencies, and only stagnates for low frequencies, 
  but is never divergent.

A precise estimate of how fast the classical Schwarz method applied
to the time harmonic Navier equations diverges
depending on the overlap is given by the following theorem:
\begin{theorem}[Asymptotic convergence factor]\label{ClassicalSchwarzAsymptotics}
The maximum of the convergence factor of the classical Schwarz method
(\ref{ClassicalSchwarz}) applied to the Navier equations
(\ref{NavierEq}) behaves for small overlap $\delta$ asymptotically as
$$
\underset{k}{\max}( \max|r_\pm|) = \textstyle 1 + \dfrac{ \sqrt{2}C_s \omega\left(3 C_p^2-\sqrt{C_p^4+8 C_s^4}\right) \sqrt{C_p^2 \sqrt{C_p^4+8 C_s^4} - C_p^4 - 2 C_s^4} }{ C_p (C_p^2+C_s^2)^{\frac{3}{2}} \left( \sqrt{C_p^4+8 C_s^4}-C_p^2\right)} \delta.
$$
\end{theorem}
\begin{proof}
According to Corollary \ref{ClassicalSchwarzTheorem}, the
  maximum of the convergence factor is attained on the interval where
  the algorithms is divergent, $k
    \in\left(\frac{\omega}{C_p},\frac{\omega}{C_s}\right)$, and this
  quantity is larger than one. For a fixed $k$, the convergence
  factor in that region is for overlap $\delta$ small given by
\begin{equation}\label{RhoMaxAsympt}
 \rho_{cla}(k,\omega,C_p,C_s,\delta) = 1 + \dfrac{2 \omega^2 \lambda_2 \bar{\lambda}_1^2 }{C_p^2\left(k^4+\bar{\lambda}_1^2 \lambda_2^2\right)}\delta+\mathcal{O}\left(\delta^2\right)\in\mathbb{R_+^*}.
\end{equation}
Computing the maximum of \eqref{RhoMaxAsympt}, we find the
result of the theorem.
\end{proof}

\section{Numerical experiments}\label{sec:3}

We illustrate now the divergence of the classical Schwarz
  algorithm with a numerical experiment. We choose the same
parameters $C_p=1$, $C_s=0.5$, $\rho=1$ and overlap
$\delta=0.1$ as in Figure \ref{Fig1}. We discretize the
time-harmonic Navier equations using $P1$ finite elements on the
domain $\Omega=(-1,1)\times(0,1)$ and decompose the domain into
two overlapping subdomains $\Omega_1=(-1,2h)\times(0,1)$ and
$\Omega_2=(-2h,1)\times(0,1)$ with $h=\frac{1}{40}$, such that the
  overlap $\delta=0.1=4h$. Our computations are performed with
the open source software Freefem++.  We show in Figure \ref{FigNumExp}
the error in modulus at iteration
\begin{figure}[t]
  \centering
 \includegraphics[width=0.49\textwidth,clip]{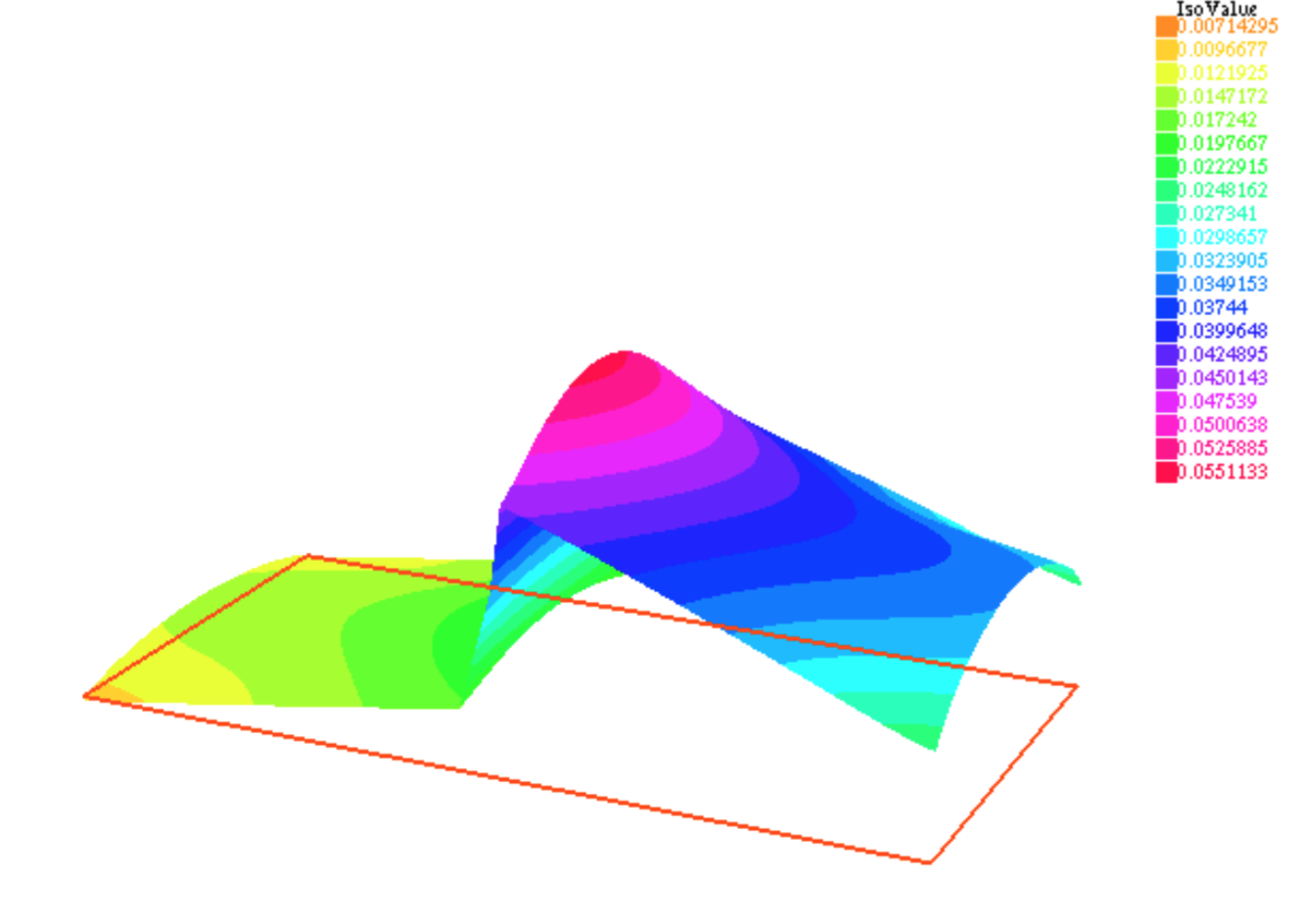}
  \includegraphics[width=0.49\textwidth,clip]{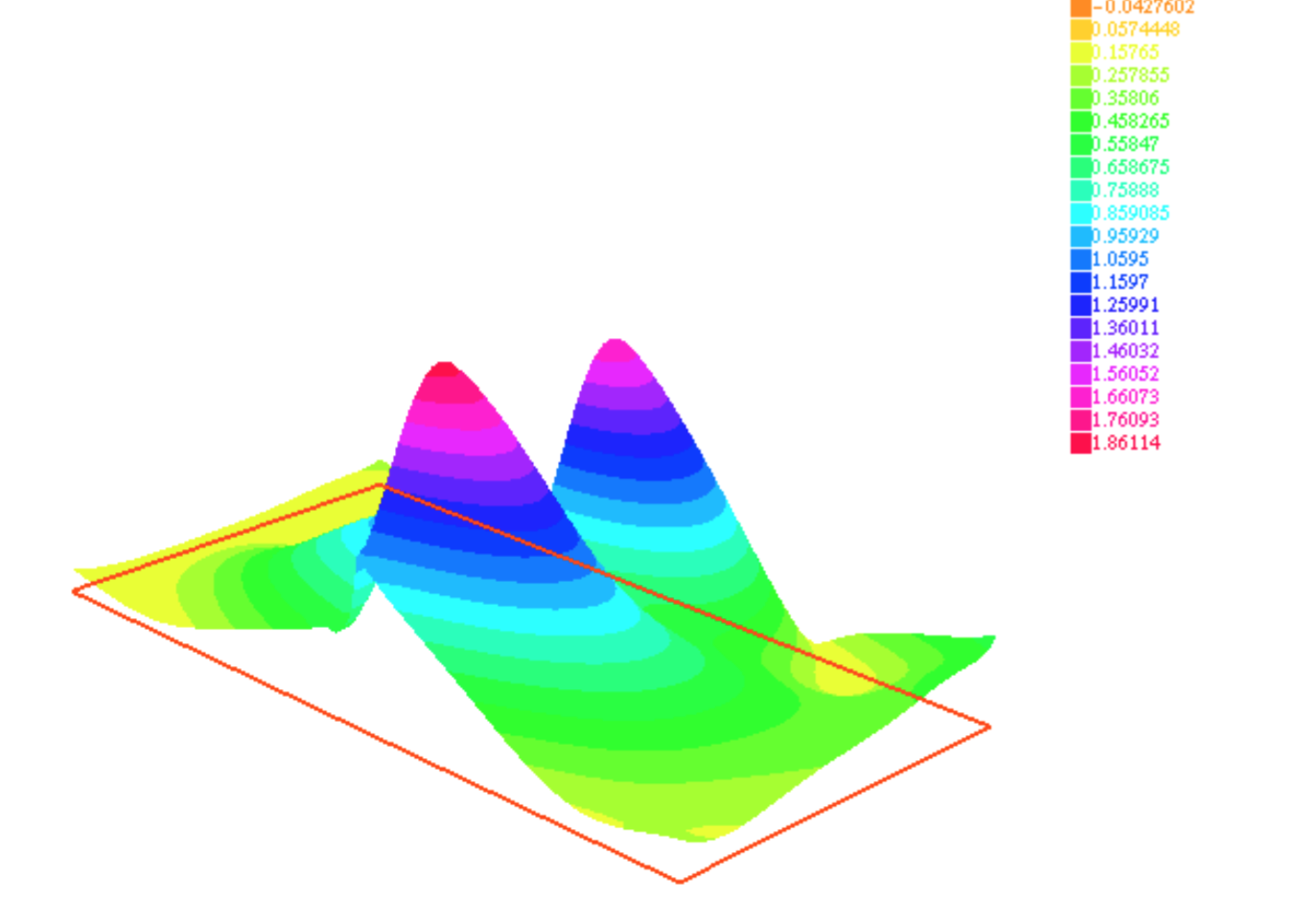}
  \caption{Error in modulus at iteration 25 of the classical Schwarz
    method with 2 subdomains, where one can clearly identify the
    dominant mode in the error: Left: $\omega=1$. Right: $\omega=5$.}
  \label{FigNumExp}
\end{figure}
$25$ of the classical Schwarz method, on the left for $\omega=1$ and
on the right for $\omega=5$. In the first case, $\omega=1$, we
  observe very slow convergence, the error decreases from $7.89e-1$ to
  $5e-2$ after 25 iterations.  This can be understood as follows: the
  lowest frequency along the interface on our domain $\Omega$ is
  $k=\pi$, which lies outside the interval
$\left[\frac{\omega}{C_p},\frac{\omega}{C_s}\right] = [1,2]$ of
frequencies on which the method is divergent. The method thus
  converges, all frequencies lie in the convergent zone in the plot in
  Figure \ref{Fig1} on the left where $\rho_{cla}<1$. The most slowly convergent mode is
  $|\sin(ky)|$ with $k=\pi$, which is clearly visible in Figure
  \ref{FigNumExp} on the left. This is different for $\omega=5$,
  where we see in Figure \ref{FigNumExp} on the right the dominant
  growing mode. The interval of frequencies on which the method is
divergent is given by
$\left[\frac{\omega}{C_p},\frac{\omega}{C_s}\right] = [5,10]$, and we
clearly can identify in Figure \ref{FigNumExp} on the right a mode
  with two bumps along the interface, which corresponds to the mode
$|\sin(ky)|$ along the interface for $k=2\pi\approx 6$, which is the fastest diverging mode according to the analytical
  result shown in Figure \ref{Fig1} on the right.

One might wonder if the classical Schwarz method is nevertheless
a good preconditioner for a Krylov method, which can happen also
  for divergent stationary methods, like for example the Additive
Schwarz Method applied to the Laplace problem, which is also not
  convergent as an iterative method \cite{efstathiou2003restricted},
  but useful as a preconditioner. To investigate this, it suffices to
  plot the spectrum of the identity matrix minus the iteration
operator in the complex plane, which corresponds to the
preconditioned systems one would like to solve.  We see in Figure
  \ref{Fig2} that the part of the spectrum that leads to a
contraction
\begin{figure}[t]
\centering
\includegraphics[width=0.49\textwidth]{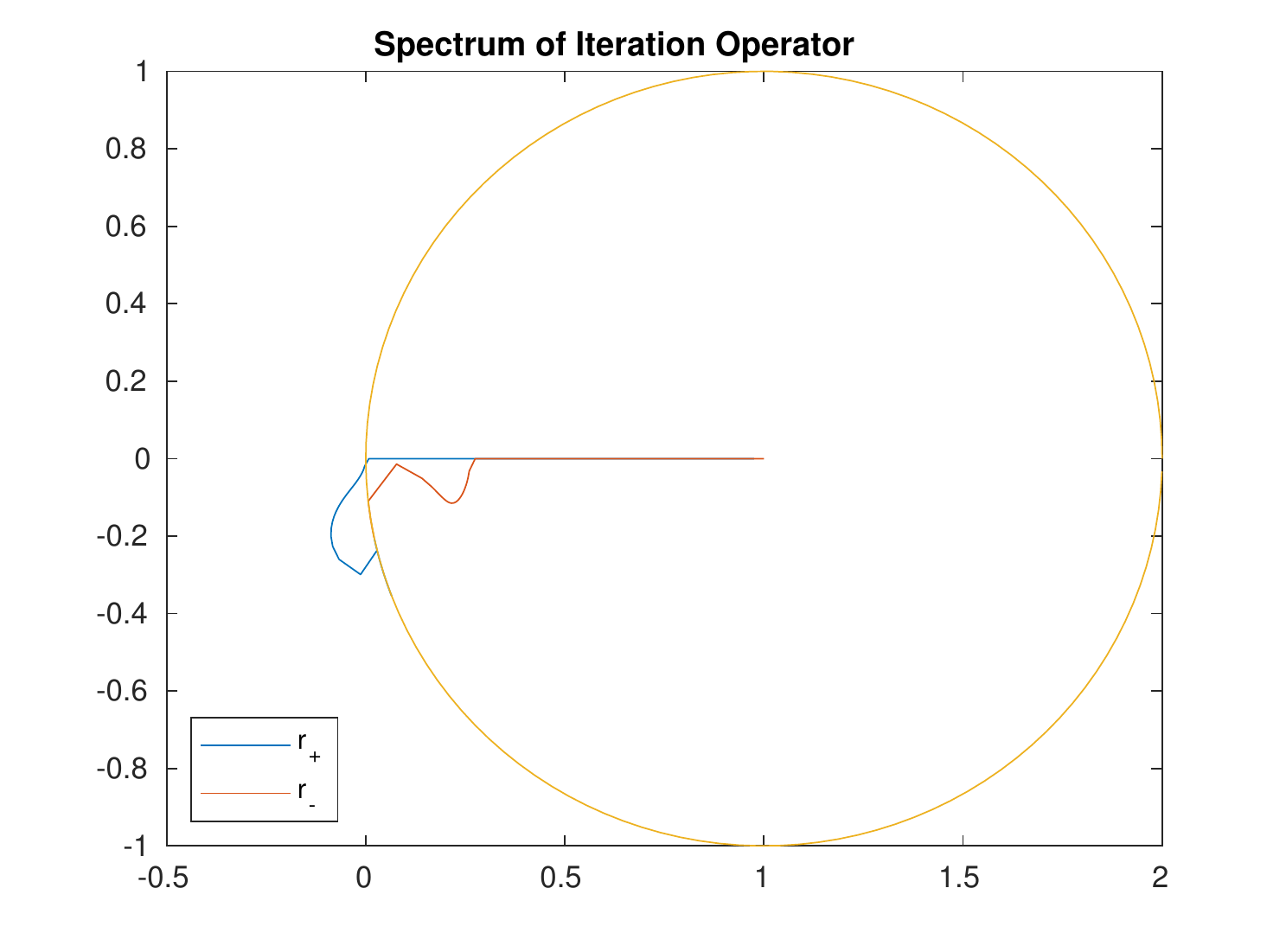}
\includegraphics[width=0.42\textwidth]{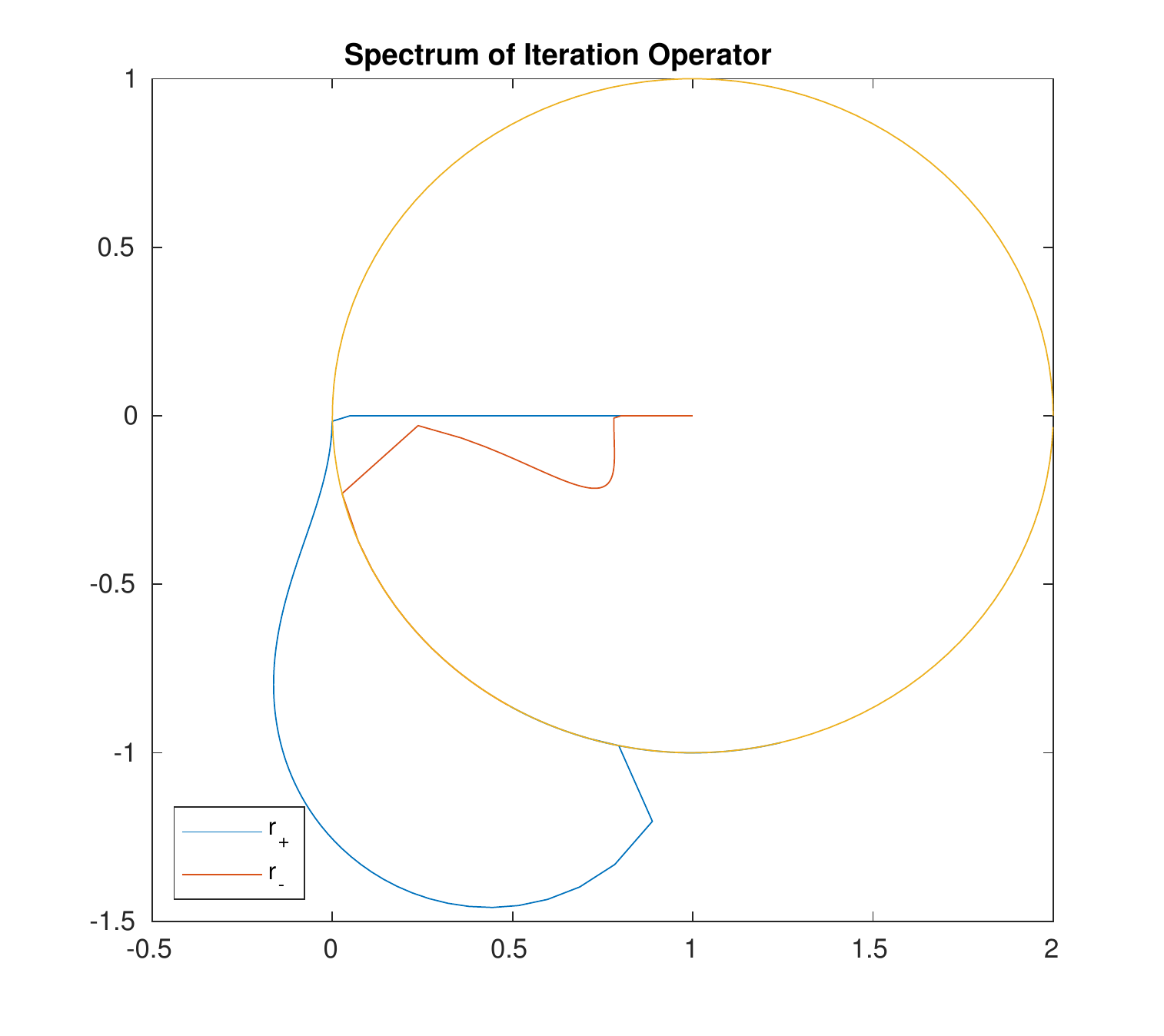}
  \caption{Spectrum of the iteration operator for the same example as
    in Figure \ref{Fig1}, together with a unit circle centered around
    the point $(1,0)$.  Left: $\omega=1$. Right: $\omega=5$}
  \label{Fig2}
\end{figure}
factor $\rho_{cla}$ with modulus bigger than one lies unfortunately
close to zero in the complex plane, and that is where the residual
polynomial of the Krylov method must equal one. Therefore we can infer
that the classical Schwarz method will also not work well as a
preconditioner. This is also confirmed by the numerical results
shown in Figure \ref{Fig3},
\begin{figure}
  \centering
  \includegraphics[width=0.45\textwidth]{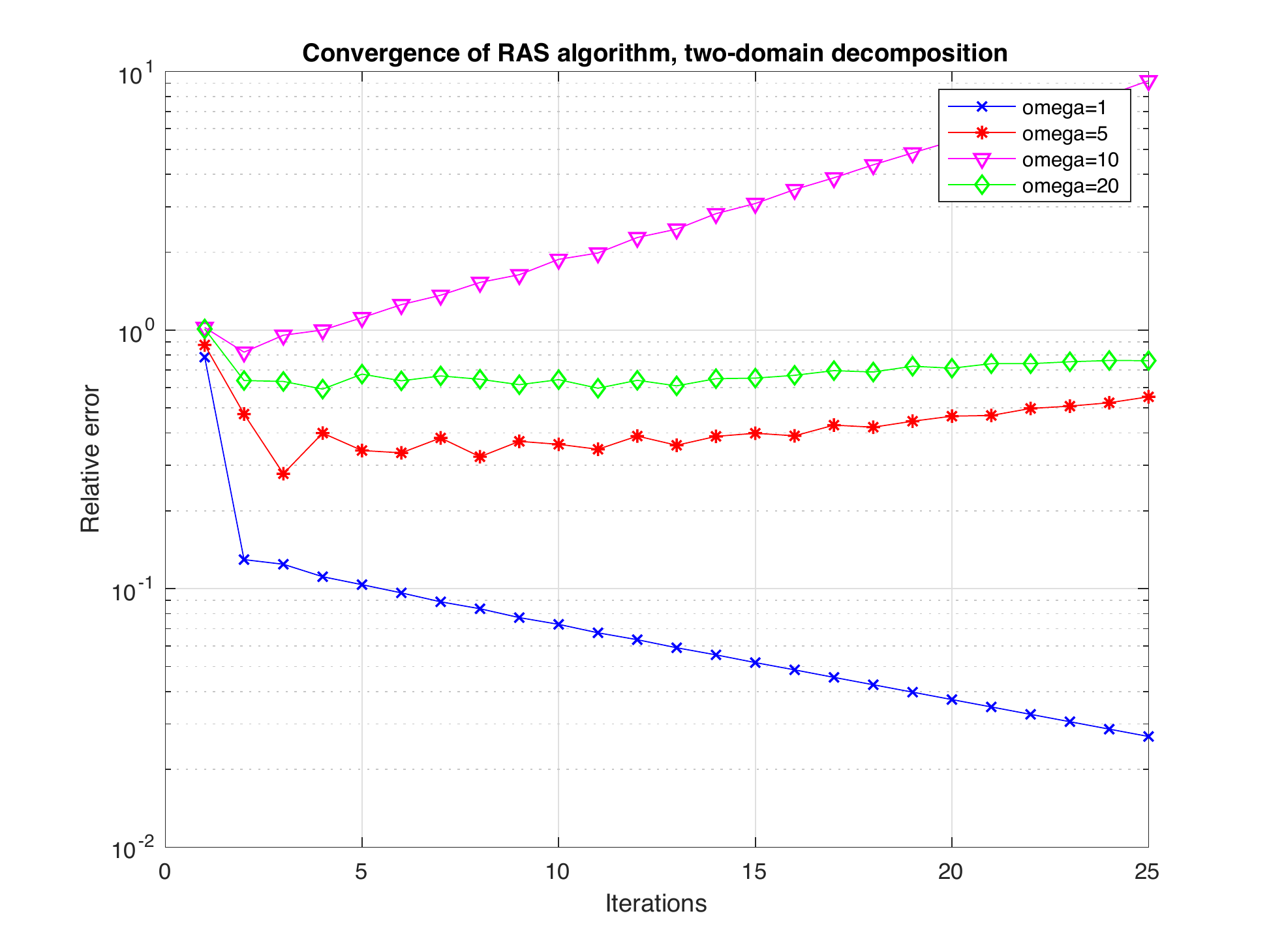}
  \includegraphics[width=0.45\textwidth]{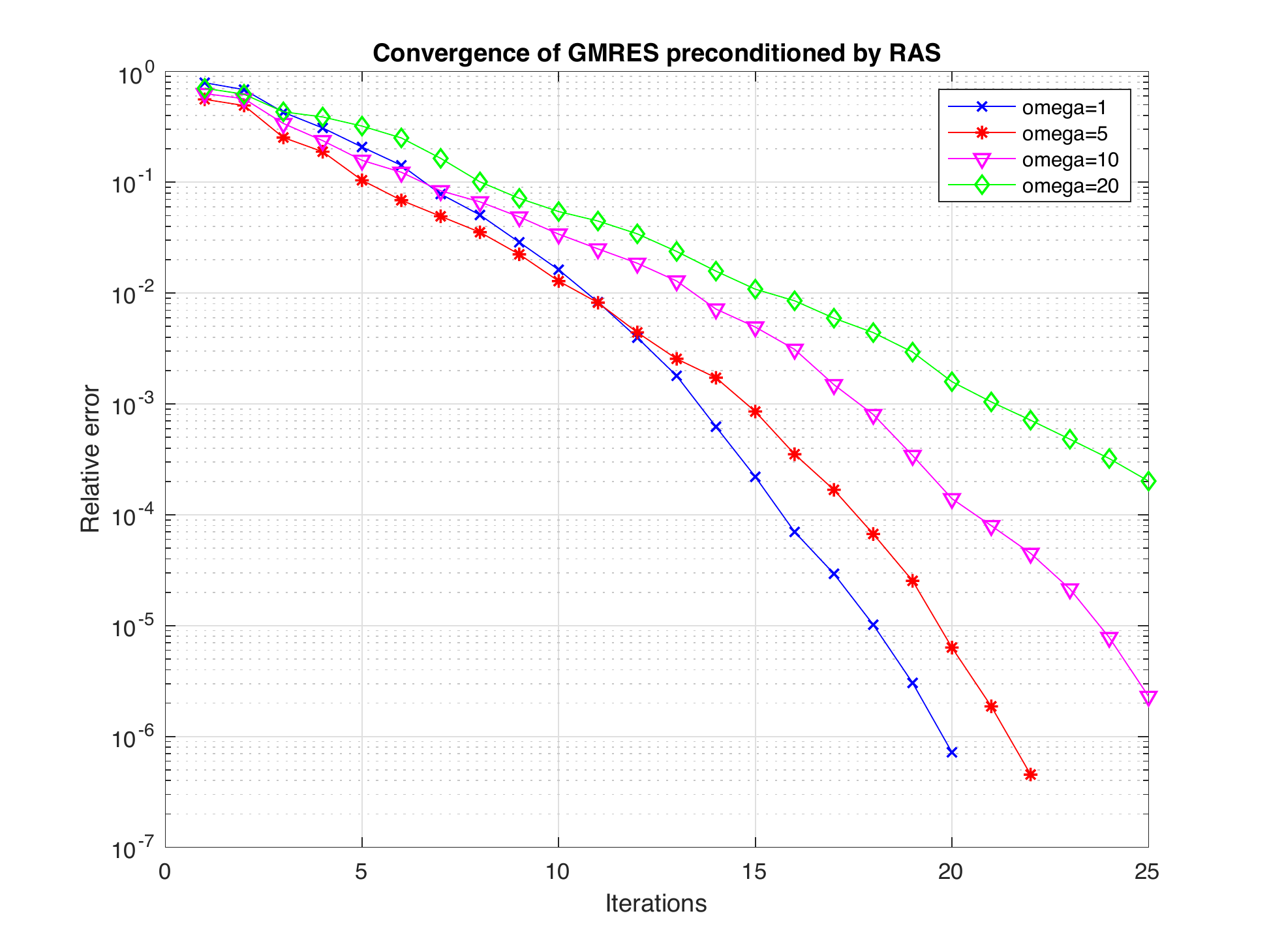}
  \caption{Convergence history for RAS and GMRES preconditioned by RAS for different values of $\omega$}
  \label{Fig3}
\end{figure}
where we used first the classical Schwarz method as a solver and then as preconditioner for GMRES. We see that GMRES now makes the method converge, but convergence depends strongly on $\omega$ and slows down when $\omega$ grows.

\section{Conclusion} \label{sec:4}

We proved that the classical Schwarz method with overlap
applied to the time harmonic Navier equations cannot be used as
an iterative solver, since it is not convergent in general. This
is even worse than for the Helmholtz or time harmonic Maxwell's
equations, for which the classical Schwarz algorithm also
  stagnates for all propagative modes, but at least is not
  divergent. We then showed that our analysis clearly identifies the
  problematic error modes in a numerical experiment. Using the
  classical Schwarz method as a preconditioner for GMRES then
  leads to a convergent method, which however is strongly dependent on
  the time-harmonic frequency parameter $\omega$. We are currently
  studying better transmission conditions between subdomains, which
  will lead to optimized Schwarz methods for the time harmonic Navier
  equation.


\begin{thebibliography}{99.}%

\bibitem{Brunet:2018}
R. Brunet.
\newblock Domain decomposition methods for time-harmonic elastic
waves.
\newblock {\em PhD thesis, University of Strathclyde}, 2018.

\bibitem{Brunet:2018:OSM}
R. Brunet, V. Dolean, M.J. Gander.
\newblock Analysis of natural Schwarz algorithms and preconditioners
for the solution of time-harmonic elastic waves.
\newblock {\em paper in preparation}.

\bibitem{darrigol2002between}
O. Darrigol.
\newblock Between hydrodynamics and elasticity theory: the first five
births of the Navier-Stokes equation.
\newblock {\em Archive for History of Exact Sciences, Springer},
56(2):95--150, 2002.

\bibitem{Dolean2008Why}
V. Dolean and M.J. Gander.
\newblock Why Classical {S}chwarz Methods Applied to Hyperbolic
Systems Can Converge even Without Overlap.
\newblock {\em Domain decomposition methods in science and engineering
  XVII}, Springer, 467--475, 2008.
  
\bibitem{vic2009}
V. Dolean, M.J. Gander, and L. Gerardo-Giorda.
\newblock Optimized Schwarz methods for Maxwell's equations.
\newblock {\em SIAM J. Sci. Comput.}, vol 31, pp. 2193-2213, 2009.

\bibitem{efstathiou2003restricted}
E. Efstathiou and M.J. Gander.
\newblock Why restricted additive {S}chwarz converges faster than
additive {S}chwarz.
\newblock {\em BIT. Numerical Mathematics}, vol 43, pp. 945-959, 2003.

\bibitem{eringen1977elastodynamics}
A.C. Eringen, and E.S. Suhubi, 
\newblock Elastodynamics, Vol. 2.
\newblock \textit{ACADEMIC PRESS New York San Francisco London}, 1977.

\bibitem{ernst2012difficult}
O.G. Ernst, M.J. Gander.
\newblock Why it is difficult to solve {H}elmholtz problems with
classical iterative methods.
\newblock {\em Numerical analysis of multiscale problems, Springer},
pp. 325-363, 2012.

\bibitem{gander2006optimized}
M.J. Gander.
\newblock Optimized {S}chwarz methods.
\newblock {\em SIAM Journal on Numerical Analysis}, 44(2):699--731, 2006.

\bibitem{huttunen2004ultra}
T. Huttunen, P. Monk, F. Collino, and P. Kaipio~Jari.
\newblock The ultra-weak variational formulation for elastic wave problems.
\newblock {\em SIAM Journal on Scientific Computing}, 25(5):1717--1742, 2004.

\bibitem{lautrup2011physics}
B. Lautrup.
\newblock Physics of continuous matter: exotic and everyday phenomena
in the macroscopic world.
\newblock {\em CRC press}, 2011.

\end{thebibliography}
\end{document}